\documentclass[11pt,a4,epsf]{article}
\usepackage{amssymb,amsmath,graphicx,epsfig}
\usepackage{latexsym}
\usepackage{enumerate}
\usepackage{color}
\usepackage{amsthm}

\usepackage{graphicx}
\usepackage{listings}

\usepackage{subcaption}

\usepackage{rotating}
\usepackage{float}
\usepackage{wrapfig}

\usepackage{xcolor}

\newtheorem{theorem}{\bf {Theorem}}

\newtheorem{example}{\bf {Example}}

\newtheorem{proposition}{\bf {Proposition}}

\title{\Large\bf Mathematical Programming formulations for the efficient solution of the $k$-sum approval voting problem}
\author{Diego Ponce$^{1}$, Justo Puerto$^{1}$, Federica Ricca$^{2}$, Andrea Scozzari$^3$\\
\\
\normalsize $^1$IMUS and Dep. Estad{\'\i}stica e Investigaci\'on Operativa, \\ Universidad de Sevilla\\
\normalsize dponce@us.es, puerto@us.es\\
\normalsize $^2$Universit\'a di Roma, La Sapienza\\
\normalsize federica.ricca@uniroma1.it\\
\normalsize $^3$Universit\'a degli Studi Niccol\'o Cusano, Roma\\
\normalsize andrea.scozzari@unicusano.it\\}

\begin{document}
\maketitle

\begin{abstract}
In this paper we address the problem of electing a committee among a set of $m$ candidates and on the basis of the preferences of a set of $n$ voters. We consider the approval voting method in which each voter can approve as many candidates as she/he likes by expressing a preference profile (boolean $m$-vector). In order to elect a committee, a voting rule must be established to `transform' the $n$ voters' profiles into a winning committee. The problem is widely studied in voting theory; for a variety of voting rules the problem was shown to be computationally difficult and approximation algorithms and heuristic techniques were proposed in the literature. In this paper we follow an Ordered Weighted Averaging approach and study the $k$-sum approval voting (optimization) problem in the general case $1 \leq k <n$. For this problem we provide different mathematical programming formulations that allow us to solve it in an exact solution framework. We provide computational results showing that our approach is efficient for medium-size test problems ($n$ up to 200, $m$ up to 60) since in all tested cases it was able to find the exact optimal solution in very short computational times.
\end{abstract}

\section{Introduction}\label{Sec1}
A typical problem in collective decision making is to select one (or more) winners among a set of $m$ candidates on the basis of the vote of a set of $n$ voters. This situation arises in many different real-life contexts, as in sport competitions to select the set of winners, or in political elections, and, more in general, whenever a committee must be formed from a larger set of candidates to represent the voters (for example, in companies or universities). For $m>1$ the problem is a {\em multi-winner} one and the ballot gives the possibility to each voter to express her/his preference for each single candidate by {\em approving} or not her/his nomination. This means that voters can approve as many candidates as they like by their preference profile ({\em approval balloting}).
Approval voting is a well-known method used for this kind of multi-winner elections. The method, introduced by Brams and Fishburn in 1978 \cite{BramsFishburn1978}, was widely studied in the literature on voting theory (see \cite{BramsFishburn2005, BramsFishburn2007, LaslierSanver2010} and the references therein).

Consider a set $N$ of $n$ voters and a set $A$ of $m$ candidates. For each $i$, $P_i=(p_{i1}, \ldots, p_{ij}, \ldots, p_{im})$ denotes the preference profile of voter $i$ which corresponds to a boolean $m$-vector whose generic element $p_{ij}$ is equal to $1$ if candidate $j$ is approved by $i$ and equal to $0$ otherwise. Therefore in the problem input we have also a set $P$ of preference profiles of the voters, so that a generic instance can be denoted by $(N,A,P)$. The problem is to find in $A$ the ``best'' subset of candidates (\emph{winning} or \emph{elected} committee), according to a certain voting rule (criterion).

When we have a single-winner election, the most accepted voting rule is to elect the candidate that has been the most approved, i.e., the one which received the largest number of votes (with a tie-breaking mechanism if needed).\\
For multi-winner elections, several voting rules have been proposed for approval voting \cite{Kilgour2010}. For a majority of the voting rules computing a winning committee is a difficult problem \cite{Aziz2015,FishburnPekec2004}. Among the many, there is a class of rules known as \emph{centralization procedures} that was widely studied in the literature. Two rules in this class were mainly analyzed, namely, one based on the {\em minisum} criterion and one on the {\em minimax} criterion. According to the first criterion, the winning committee corresponds to the subset of candidates in $A$ that minimizes the sum of the $n$ Hamming distances to the preferences profiles of the $n$ voters. The second criterion selects the subset of candidates that minimizes the maximum of its Hamming distances to the voters' profiles. Recently, a new family of rules has been proposed to generalize minisum and minimax in \cite{Amanatidis_et_al2015}, where the authors introduce a family of voting rules which makes use of {\em Ordered Weighted Averaging} operators (OWA) \cite{Yager1998}. In this setting, a vector of $n$ weights $W=(w_1, \ldots, w_n)$ is fixed; then the $n$ distances between voters' profiles and the decision vector (committee) are ordered from largest to smallest and they are weighted with the corresponding weight in $W$. Clearly, when $W=(1, 0, \ldots, 0)$ we have the minimax criterion, while for $W=(\frac{1}{n}, \ldots, \frac{1}{n})$ we have the minisum criterion. Many other criteria can be defined in this way by tuning the weights in $W$ according to the specific goal of the application. An interesting class of problems arises when vector $W$ has only $0/1$ values and more than one weight equal to 1. Suppose to have $k$ elements equal to 1; when they are in the first $k$ positions of the vector of weights they refer to the $k$ largest distances, thus providing what is known in the literature as the $k$-sum approach already applied to many other combinatorial problems \cite{NickelPuerto2005}. We have a similar problem when we have elements equal to 1 in the last $h$ positions of the weighting vector ($h$ smallest distances). Both problems have meaningful applications in approval voting.

In this paper we study the problem of selecting a committee by applying approval voting and basing on a $k$-sum objective function ($k$-sum approval voting problems). In \cite{Amanatidis_et_al2015} it is proved that for $1 \leq  k < n$ the problem is NP-hard and, therefore, an approximation algorithm is provided to get feasible solutions with guaranteed bounds. On the other hand, the same authors provide polynomial time exact algorithms for some families of weighting vectors that consider the $h$ smallest distances ($h$ fixed). In the present paper we study these kind of problems under a mathematical programming viewpoint, providing different exact formulations for the $k$-sum approval voting problem, with $1 \leq k < n$. We then exploit these formulations to develop exact solution procedures that may be used to solve medium-size problems at optimality, or to efficiently find a sub-optimal solution when the size of the problem is too large. To develop such formulations we rely on the general approach for solving $k$-sum optimization problems provided in \cite{OT03,PuertoRodriguezChiaTamir2017} and on results in \cite{PuertoBlanco2014}. We experimentally study the solution of our $k$-sum approval voting problem by using the above formulations in a Branch \& Bound framework. All formulations were tested on a variety of medium-size randomly generated test problems, in all cases providing the exact optimum in very short times. In view of this, our formulations also provide an efficient tool to certify optimality of a solution.

We apply the mathematical programming approach also when the $h$ smallest distances are considered in the objective function. We formulate this problem as a polynomial sequence of linear program, thus providing a formal proof that it can be solved in polynomial time as already established in \cite{Amanatidis_et_al2015}.

The paper is organized as follows. Section \ref{Sec2} formally defines the problem and sets the notation. Section \ref{Sec_2} presents our mathematical programming formulations for the $k$-sum approval voting problem. We have developed two different types of formulations. The first ones based on an exponential number of constraints that can be separated efficiently (see Section  \ref{sec3_1}) and the second ones based on compact representations of $k$-sums (see Section \ref{sec3_2}). In Section \ref{sec4} we describe how all the above mentioned formulations can be strengthened with variable fixing and valid inequalities. The computational experiments are reported in Section \ref{s:comp-res}. There we compare the performance of the formulations on two different set of instances showing its usefulness in solving the problem for instances of medium to large sizes. Finally, Section \ref{sec6} contains our concluding remarks.

\section{Problem definition and basic results}\label{Sec2}

Consider an instance of the $k$-sum approval voting problem $(N, A, P)$. For a given committee $x$ (i.e., a boolean vector $x$ of length $m$) we compute the Hamming distance between $x$ and each voter profile $P_i$, $i=1,\ldots,n$ thus obtaining the Hamming distance $d_i(x)$ for each voter $i$. Following the OWA approach, we consider a family of functions, parameterized by a vector of length $n$ that maps a vector of distances $(d_1(x), \ldots, d_n(x))$ to an aggregated function $D(x)$ (OWA score). The $k$-sum approval voting problem can be stated as follows: select a committee $x$ minimizing the OWA score of Hamming distances $D(x)$. In this paper we study two families of $k$-sum approval voting problems. The first computes the OWA score using the following vector of weights:

$$W(k)=(1,\ldots,1,0,\ldots,0),$$

\noindent where $k$ refers to the number of ones in the first $k$ positions of vector $W$ (electing a committee that minimizes the sum of the $k$ largest Hamming distances).

\noindent The second family uses the following weighting vector:

$$M(n-h)=(0,\ldots, 0,1,\ldots,1),$$

\noindent where $h$ weights equal to 1 are in the last $h$ positions (electing a committee that minimizes the sum of the $h$ smallest Hamming distances).
Clearly, the cases $k=h=n$ are the same, and, in this case, we have the same OWA problem, that, in fact, corresponds to the minisum problem.
\vskip 8 pt
\noindent In \cite{Brams_et_al2007} (see Proposition 4) the authors account for why the minisum problem is polynomially solvable, but they do not provide a formal proof. The key idea is that the minisum winning committee, in particular under a cardinality constraint on the size of the committee fixed to $C$, corresponds to a set of $C$ candidates receiving the most votes. In the following, using Linear Programming (LP), we give a formal proof that the problem is polynomial when $k=n$. Denote by $\gamma_j$ the number of votes for candidate $j$, $j=1,\ldots,m$, we have

$$\gamma_j=\sum\limits_{i=1}^n p_{ij}.$$

\vskip 8 pt
\noindent Consider the case when the size of the committee is not given. A valid formulation can be obtained basing on the following observation (see \cite{Brams_et_al2007}): when $k=n$, all voters' Hamming distances $(d_1(x), \ldots, d_n(x))$ are considered in the objective function, so that a candidate $j$ is elected if the number of votes for her/him $\gamma_j$ is greater than $n-\gamma_j$. This leads to the following model

\begin{align*}
\min & \; \sum_{j=1}^m (n-\gamma_j)x_j+\sum_{j=1}^m \gamma_j(1-x_j) \\
\mbox{s.t. \;} & x\in \{0,1\}^m. \nonumber
\end{align*}

\vskip 8 pt
\noindent Since the objective function is linear, the optimal solution is attained at some vertex of the $m$-dimensional hypercube. Then, we can relax the binary variables of the above problem obtaining the LP model

\begin{align*}
\min & \; \sum_{j=1}^m (n-\gamma_j)x_j+\sum_{j=1}^m \gamma_j(1-x_j) \\
\mbox{s.t. \;} & 0\le x \le 1. \nonumber
\end{align*}

\vskip 8 pt
\noindent As in \cite{Brams_et_al2007}, we can also consider a constraint on the size $C$ of the committee, obtaining the following model

\begin{equation}\label{LP}
\begin{array}{lcl}
\min & \; \sum\limits_{j=1}^m (n-\gamma_j)x_j+\sum\limits_{j=1}^m \gamma_j(1-x_j) \\
\mbox{s.t. \;} &  \sum\limits_{j=1}^m x_j = C\\
& {x\in \{0,1\}^m}.
\end{array}
\end{equation}

\vskip 8 pt
\noindent Since the constraint matrix is  again Totally Unimodular (TU), the problem can be still solved by Linear Programming techniques, after relaxing the binary constraints on the variables $x$ (see \cite{NemhauserWolsey1988}). This definitely shows that the minisum case is polynomially solvable.
\newline

%In particular, instead of solving model (\ref{LP}), it suffices to sort the values $\gamma_j$, $j=1,\ldots,m$, in non-increasing order, and then choose the corresponding first $C$ candidates. This can be done in $O(m\log m)$ time.

When $k < n$ the $k$-sum approval voting problem is NP-hard (see \cite{Amanatidis_et_al2015, FrancesLitman1997}). This justifies the idea of studying efficient solution methods for the general problem resorting to heuristic approaches, or approximation algorithms \cite{ByrkaSornat2010,Caragiannis_et_al_2010,LeGrand_et_al_2002,Li_et_al_2002}. In the following sections we present a number of exact mathematical  programming formulations of the general $k$-sum approval voting problem with $1\leq k<n$ that can be efficiently solved in a Branch \& Bound framework enriched by the use of valid inequalities.

A similar problem arises when we consider the weighting vector $M(n-h)$ with $1\leq h<n$. The computational complexity of this problem was established in \cite{Amanatidis_et_al2015} when $h$ is not part of the input of the problem. Even if it is shown that the problem is polynomially solvable in this case, to the best of our knowledge, the computational complexity is still open when the problem is formulated with a general $h$. We discuss this problem in the final section of the paper.

\section{Mathematical Programming formulations for the Approval Voting problem}\label{Sec_2}

\noindent Consider now $W(k)$ with $1\leq k<n$. Let $x$ be a committee, the Hamming distance, $d_i(x)$, between the profile $P_i$ of voter $i$ and $x$ is given by

$$d_i(x)=\sum\limits_{j=1}^m |p_{ij}-x_j|.$$

\noindent Since both $P_i$, $i=1,\ldots,n$, and $x$ are boolean vectors, we can exploit the fact that the Hamming distance between two boolean vectors corresponds to the vector of the \emph{Exclusive-or} between the vector elements \cite{GivantHalmos2009}. Thus the Hamming distance $d_i(x)$ can be rewritten as follows

\begin{equation}\label{New_Hamming}
\begin{array}{lcl}
&d_i(x)=\sum\limits_{j=1}^m z_{ij}\\
& z_{ij}\geq x_j-p_{ij}x_j&\;j=1,\ldots,m  \\
& z_{ij}\geq p_{ij}-p_{ij}x_j&\;j=1,\ldots,m.
\end{array}
\end{equation}

\vskip 8 pt
We can also replace the two inequalities above by the equivalent representation:
$ z_{ij}\geq x_j(1-p_{ij})+p_{ij}(1-x_j).$
This gives rise to:
\begin{align}\label{New_Hamming1}
d_i(x)=&\sum\limits_{j=1}^m z_{ij}\\
& z_{ij}\geq x_j(1-p_{ij})+p_{ij}(1-x_j), \quad j=1,\ldots,m.
\end{align}

%\noindent Hence, given $k$, with $1<k<n$, the problem of electing a committee that minimizes the sum of the $k$ largest weights can be formulated as a Mixed Integer Linear Programming (MILP) problem for a given permutation $\sigma$ such that  $d_{\sigma_i}\ge d_{\sigma_{i+1}}, \; i=1,\ldots,n-1$,
%as follows:
%
%\begin{align}\label{Second}
%\min& \sum\limits_{h=1}^k d_{\sigma_i}\\
%& z_{ij}\geq x_j(1-p_{ij})+p_{ij}(1-x_j), \quad i=1,\ldots,n,\;j=1,\ldots,m  \\
%& d_i=\sum\limits_{j=1}^m z_{i j}, \quad i=1,\ldots,n\\
%& 0\leq z_{ij}\leq 1, \quad  i=1,\ldots,n,\;j=1,\ldots,m.\\
%& x_j\in \{0,1\}, \quad j=1,\ldots,m
%\end{align}
%\vskip 8 pt

\noindent Let $\sigma_x: \{1, \ldots,n\} \rightarrow \{1, \ldots,n\}$ be an ordering function that, for a given $x$, provides a permutation of the voters' indices such that $d_{\sigma_x(1)}(x)\geq d_{\sigma_x(2)}(x)\geq \ldots\geq d_{\sigma_x(n)}(x)$. For the given permutation the problem of electing a committee that minimizes the sum of the $k$ largest distances can be formulated as follows:

\begin{equation}\label{First}
\begin{array}{lcl}
\min& \sum\limits_{h=1}^k d_{\sigma_x(h)}(x)\\
\\
& {x\in \{0,1\}^m}.
\end{array}
\end{equation}
\vskip 8 pt

\noindent Following the approach in \cite{PuertoBlanco2014}, the above problem can be restated as:

\begin{equation}\label{general_3}
\begin{array}{lcl}
\min\limits_{x\in \{0,1\}^m} \big(\max \big\{ d_{S}(x) \; |\; S\subset \{ 1,\ldots,n \}, |S|=k \big\} \big),\\
\end{array}
\end{equation}
where $d_{S}(x) =\sum\limits_{i\in S} \sum\limits_{j=1}^m | p_{ij}-x_j|$, is the Hamming distance of the set of voters in $S$ to the committee represented by $x$.

In general $k$-sum problems, the expression of the contribution of a subset of voters to the election of a candidate can be also computed by means of the $\gamma_j$ values, namely the number of voters approving a given candidate $j$. For this purpose, let us introduce some necessary notation. More formally, let $S$ be a set of voters such that $|S|=k$. We define $\gamma_j(S)=\sum_{i\in S} p_{ij}$, as the number of votes of candidate $j$ by the voters in $S$. For a given $x$ and $S$ such that $|S|=k$, we can compute:

\begin{align}d_S(x)& =\sum\limits_{j=1}^m\max\{\gamma_j(S)(1-x_j),(k-\gamma_j(S))x_j\} \label{GenDis}\\
& =\sum_{j=1}^m \gamma_j(S)(1-x_j)+\sum_{j=1}^m (k-\gamma_j(S))x_j, \label{GenDis1}
\end{align}

\noindent i.e., the Hamming score of the  $k$ voters in $S$ computed w.r.t. a given solution $x$. Notice that by means of expressions (\ref{GenDis}) and (\ref{GenDis1}) the score is well calculated even when the solution is not optimal. Basing on these expressions, in the following sections we obtain alternative valid formulations of the $k$-sum  approval voting problem that we then test experimentally in Section \ref{s:comp-res}.

\subsection{Valid formulations based on subsets of size $k$}\label{sec3_1}

In this section we propose a first valid formulation for our $k$-sum approval voting problem which is based on expression (\ref{GenDis}). We formulate it in the following theorem where, for the sake of simplicity, we avoid specifying $S\subset \{1,\ldots,n\}$ when not necessary.

\begin{theorem} \label{th:form1}
An optimal solution of the $k$-sum approval voting problem can be obtained solving the following integer programming problem.
\begin{align}\label{GeneralProblemDis}
\min\; & v\\
\mbox{ s.t. } & z_{ij}\ge p_{ij}(1-x_{j})\quad \forall i,j  \noindent \label{GeneralProblemDis:const1}\\
&z_{ij}\ge (1-p_{ij})x_{j}\quad \forall i,j   \noindent \label{GeneralProblemDis:const2}\\
&\sum\limits_{j=1}^m\sum_{i\in S}z_{ij}\le v\quad \forall S: |S|=k \noindent \label{GeneralProblemDis:const3}\\
%& 0\leq z_{i,j}\leq 1\quad \forall i,j \noindent \label{GeneralProblemDis:const4}\\
& {x\in \{0,1\}^m}.\noindent \label{GeneralProblemDis:const5}
\end{align}
\end{theorem}
\begin{proof}
Applying (\ref{GenDis}) in formula (\ref{general_3}) gives

$$ \min_{x\in \{0,1\}^m}\left\{\sum_{j=1}^m  \max \{(k-\gamma_j(S))x_j,\gamma_j(S)(1-x_j)\}| S\subset \{1,\dots,n\}, \; |S|=k \right\}.$$

\noindent Defining variables $z_{Sj}$ for all $S\subset \{1,\ldots,n\}$, $|S|=k$, and all $j$, this is equivalent to
\begin{align}\label{GeneralProblemv1}
\min \; & v\\
\mbox{s.t. } &z_{Sj}\ge \gamma_j(S)(1-x_{j})&\forall j,\;\forall S: |S|=k  \noindent \label{GeneralProblemv1:const1}\\
&z_{Sj}\ge (k-\gamma_j(S))x_{j}&\forall j,\;\forall S: |S|=k \noindent \label{GeneralProblemv1:const2}\\
&\sum\limits_{j=1}^mz_{Sj}\le v&\forall S: |S|=k  \noindent \label{GeneralProblemv1:const3}\\
& {x\in \{0,1\}^m}. \noindent
\end{align}
Next, we can define variables $z_{ij}$ for all $i=1,\ldots,n$ and for all $j=1,\ldots,m$ and disaggregate each variable $z_{Sj}=\sum_{i\in S} z_{ij}$, which results in

\begin{align}\label{GeneralProblem}
\min\; & v\\
\mbox{ s.t. } & \sum_{i\in S}z_{ij}\ge \gamma_j(S)(1-x_{j})&\forall j,\;\forall S: |S|=k  \noindent \label{GeneralProblem:const1}\\
&\sum_{i\in S}z_{ij}\ge (k-\gamma_j(S))x_{j}&\forall j,\;\forall S: |S|=k  \noindent \label{GeneralProblem:const2}\\
&\sum\limits_{j=1}^m\sum_{i\in S}z_{ij}\le v&\forall S: |S|=k \noindent \label{GeneralProblem:const}\\
& {x\in \{0,1\}^m}.\label{GeneralProblem:const3}
\end{align}
\vskip 8 pt
\noindent We observe that constraints (\ref{GeneralProblem:const1}) and (\ref{GeneralProblem:const2}) for all $j$ and $S$ such that $|S|=k$, can be replaced by the following disaggregated version
\begin{align*}
\min\; & v\\
\mbox{ s.t. } & z_{ij}\ge p_{ij}(1-x_{j})\quad \forall i,j  \noindent \\
&z_{ij}\ge (1-p_{ij})x_{j}\quad \forall i,j   \noindent \\
&\sum\limits_{j=1}^m\sum_{i\in S}z_{ij}\le v\quad \forall S: |S|=k \noindent \\
& {x\in \{0,1\}^m}.\nonumber
\end{align*}
This concludes the proof.
\end{proof}

\begin{example}
We illustrate the above approach reformulating the minimax approval voting problem ($k=1$) within this general framework.

\begin{equation}\label{CenterVoting}
\begin{array}{lcl}
\min\quad v\\
\\
&z_{ij}\ge p_{ij}(1-x_{j})&\forall i,j\\
&z_{ij}\ge (1-p_{ij})x_{j}&\forall i,j\\
&\sum\limits_{j=1}^mz_{ij}\le v&\forall i\\
& {x\in \{0,1\}^m}.
\end{array}
\end{equation}
\end{example}

In the following, we develop a second valid formulation for the general $k$-sum approval voting problem applying (\ref{general_3}) but using (\ref{GenDis1}) to represent the Hamming distance instead of (\ref{GenDis}).

\begin{theorem} The following formulation provides a valid representation of the $k$-sum approval voting problem.
\begin{align}
\min & \; v \label{GeneralProblemDis1}\\
\mbox{s.t. \;} & \sum_{j=1}^m (k-\gamma_j(S))x_j+\sum_{j=1}^m \gamma_j(S)(1-x_j)\le v \quad  \forall S: |S|=k \label{GeneralProblemDis1:const1}\\
& x\in \{0,1\}^m. \label{GeneralProblemDis1:const2}
\end{align}
\end{theorem}
\begin{proof}
Applying in formula (\ref{general_3}) the representation (\ref{GenDis1}) instead of  (\ref{GenDis}), the proof follows similarly to that of Theorem \ref{th:form1}.
\end{proof}

\noindent Since both formulations (\ref{GeneralProblemDis})-(\ref{GeneralProblemDis:const5}) and (\ref{GeneralProblemDis1})-(\ref{GeneralProblemDis1:const2}) have an exponential number of constraints, we propose here two different approaches to solve these two models.
\newline

A first approach is based on formulation (\ref{GeneralProblemDis})-(\ref{GeneralProblemDis:const5}).  Let us assume that we embed that formulation in a Branch and Bound scheme and let $(\hat x, \hat z, \hat v)$ be the current solution in a node of this Branch \& Bound tree.

\begin{enumerate}
\item[]\textbf{Procedure for (\ref{GeneralProblemDis})-(\ref{GeneralProblemDis:const5})}
\begin{itemize}
\item Compute $\hat r_i:=\sum_{j=1}^m \hat z_{ij}$ for every $i$ and choose the $k$ largest. Determine $S$ according to the $k$ largest $\hat r_i$ for $i\in \{1,\ldots,n \}$.

\item Check if  $\sum_{i\in S}\hat r_i>\hat v$. In the affirmative case, we need to add the following constraint related to such $S$
\begin{equation}\label{Cutting plane}
\sum_{i\in S} \sum\limits_{j=1}^m z_{ij}\le v.
\end{equation}
Otherwise, i.e. when the answer to the test is no,  the current solution is feasible in the current node. Therefore,  a valid description of the problem was already available and no more inequalities have to be added.
\end{itemize}
\end{enumerate}

A similar scheme can be applied to Problem (\ref{GeneralProblemDis1})-(\ref{GeneralProblemDis1:const2}). Let $(\hat x,\hat v)$ be a given feasible solution in a node of its Branch \& Bound tree.

\begin{enumerate}
\item[] \textbf{Procedure for (\ref{GeneralProblemDis1})-(\ref{GeneralProblemDis1:const2})}
\begin{itemize}
\item Compute  $\hat r_{i}:=\sum_{j=1}^m |\hat x_j -p_{ij}|$, for all $i=1,\dots,n$. Determine $S$ according to the $k$ largest values of $\hat r_i$ for $i\in \{1,\ldots,n \}$.
    \item Check whether $\sum_{i\in S} \hat r_i >\hat v$. In the affirmative case we need to add the following inequality which is a valid cut that separates $(\hat x,\hat v)$
    $$ \sum_{j=1}^m \gamma_j(S)(1-x_j)+\sum_{j=1}^m (k-\gamma_j(S))x_j \le v. $$
\end{itemize}
\end{enumerate}

\begin{proposition}
Formulation (\ref{GeneralProblemDis1})-(\ref{GeneralProblemDis1:const2}) is at least as good as formulation (\ref{GeneralProblemDis})-(\ref{GeneralProblemDis:const5}).
\end{proposition}
\begin{proof}
Let $P_{(\ref{GeneralProblemDis1})-(\ref{GeneralProblemDis1:const2})}$ and $P_{(\ref{GeneralProblemDis})-(\ref{GeneralProblemDis:const5})}$ be the polyhedra defining the feasible domains of the continuous relaxation of formulations (\ref{GeneralProblemDis1})-(\ref{GeneralProblemDis1:const2}) and  (\ref{GeneralProblemDis})-(\ref{GeneralProblemDis:const5}), respectively.

Consider a feasible solution (possibly fractional) $(x,v,z)\in P_{(\ref{GeneralProblemDis})-(\ref{GeneralProblemDis:const5})}$. It follows that its projection onto $(x,v)$ belongs to $P_{(\ref{GeneralProblemDis1})-(\ref{GeneralProblemDis1:const2})}$, and the result follows.
\end{proof}

 The above result, together with the fact that formulation (\ref{GeneralProblemDis1})-(\ref{GeneralProblemDis1:const1}) has a smaller number of constraints and variables than formulation (\ref{GeneralProblemDis})-(\ref{GeneralProblemDis:const5}), justifies that in our computational experiments (see Section \ref{s:comp-res}) we only report results based on formulation (\ref{GeneralProblemDis1})-(\ref{GeneralProblemDis1:const1}) since its performance is superior to the one of (\ref{GeneralProblemDis})-(\ref{GeneralProblemDis:const5}).

%The above result, together with the fact that formulation (\ref{GeneralProblemDis1})-(\ref{GeneralProblemDis1:const2}) has a smaller number of constraints and variables than formulation (\ref{GeneralProblemDis})-(\ref{GeneralProblemDis:const5}), justifies the results found in our computational experiments (see Section \ref{s:comp-res}) showing that the performance of formulation (\ref{GeneralProblemDis1})-(\ref{GeneralProblemDis1:const2}) is superior to the one of (\ref{GeneralProblemDis})-(\ref{GeneralProblemDis:const5}).
\bigskip

\subsection{Valid formulations based on Hamming distance among profiles}\label{sec3_2}

In this section, we derive alternative valid formulations for the $k$-sum approval voting problem that do not make explicit use of all the possible subsets of $\{1,\ldots,n\}$ of cardinality $k$. For an arbitrary subset $S$ of $\{1,\ldots,n\}$, we consider the sum of the Hamming distances of all $P_i$, $i\in S$, to $x$ as follows

$$d_S(x)=\sum_{i\in S}\sum_{j=1}^m |x_j-p_{ij}|.$$

\noindent If there is no confusion, in the following, we simply write $d_i$ in place of $d_i(x)$.
\bigskip

\noindent For a given $k$, with $1\leq k<n$, the problem of electing a committee that minimizes the sum of the $k$ largest Hamming distances can be formulated as a Mixed Integer Linear Programming (MILP) problem, provided that for any given $x$ a permutation $\sigma_x$ such that $d_{\sigma_x(i)}(x)\ge d_{\sigma_x(i+1)}(x), \; i=1,\ldots,n-1$, is fixed.

\begin{align}\label{Second}
\min& \sum\limits_{h=1}^k d_{\sigma_x(h)}(x) \\
& z_{ij}\geq x_j(1-p_{ij})+p_{ij}(1-x_j) \quad i=1,\ldots,n,\;j=1,\ldots,m \noindent \label{Second:Const1}  \\
& d_i\geq\sum\limits_{j=1}^m z_{i j} \quad i=1,\ldots,n \noindent \label{Second:Const2}\\
%& 0\leq z_{ij}\leq 1 \quad  i=1,\ldots,n,\;j=1,\ldots,m \noindent \label{Second:Const3} \\
& x_j\in \{0,1\} \quad j=1,\ldots,m. \noindent \label{Second:Const4}
\end{align}
\vskip 8 pt

Problem (\ref{Second})-(\ref{Second:Const4}) has $m$ binary variables, $nm$ continuous variables, and $O(nm)$ linear constraints.
\vskip 8 pt

This formulation is correct but it is not operational, since it depends on the valid permutation function $\sigma_x(\cdot)$ that sorts the distances in a non-increasing order. However, it is still possible to derive alternative valid formulations that do not make explicit use of that permutation (see  \cite{OT03,PuertoRodriguezChiaTamir2017}). Indeed, let us consider a new variable $t\ge 0$ and a set of $n$ variables $v_i$, $i=1,\ldots,n$.

\begin{theorem}
The following is a valid formulation for the general $k$-sum approval voting problem.
\begin{align}
\min\; & kt+ \sum_{i=1}^n v_i &  \label{for:obj}\\
\mbox{s.t. } & v_i\ge d_i-t & i=1,\ldots,n \label{k-centrum} \\
& d_i\ge \sum_{j=1}^m z_{ij} & i=1,\ldots,n\label{Hdistance0} \\
& z_{ij}\geq x_j(1-p_{ij})+p_{ij}(1-x_j) & i=1,\ldots,n,\;j=1,\ldots,m  \label{Hdistance1}\\
%& z_{ij}\geq x_j-p_{ij}x_j&i=1,\ldots,n,\;j=1,\ldots,m \label{Hdistance1} \\
%& z_{ij}\geq p_{ij}-p_{ij}x_j&i=1,\ldots,n,\;j=1,\ldots,m \label{Hdistance2}  \\
& x_j\in \{0,1\} & j=1,\ldots,m \label{x_var} \\
%& 0\leq z_{ij}\leq 1 & i=1,\ldots,n,\;j=1,\ldots,m \label{z_var}\\
& t\ge 0, \; v_i\ge 0 &  i=1,\ldots,n. \label{Hdistance2}
\end{align}
\end{theorem}
\begin{proof}
Consider the general formulation (\ref{general_3}). Following the proof in \cite{PuertoRodriguezChiaTamir2017}, the inner maximum in problem (\ref{general_3}) is equivalent to the following

\begin{equation}\label{Inner}
\begin{array}{lcl}
\max\; & \sum\limits_{i=1}^n d_iq_i &\\
& \sum\limits_{i=1}^n q_i=k\\
& q_i\in \{0,1\} & i=1,\ldots,n.\\
\end{array}
\end{equation}

\noindent The above constraint matrix is TU and thus the integrality constraints on the variables $q_i$, $i=1,\ldots,n$, in problem (\ref{Inner}) can be relaxed to $0\leq q_i\leq 1$, $i=1,\ldots,n$, and the resulting problem has the following exact dual:

\begin{equation}\label{InnerDual}
\begin{array}{lcl}
\min\; & kt+ \sum\limits_{i=1}^n v_i  &\\
\mbox{s.t. } & v_i\ge d_i-t & i=1,\ldots,n \\
& v_i\geq 0 & i=1,\ldots,n.\\
\end{array}
\end{equation}

\noindent Notice that, since $d_i$ are distances, the coefficients in the objective function of (\ref{Inner}) are non negative and, w.l.o.g., we can set the variable $t$ as $t\geq 0$. To complete the proof, it suffices to add to the above dual problem the constraints (\ref{Hdistance0})-(\ref{x_var}).
\end{proof}
%\begin{align}
%\min\; & kt+ \sum_{i=1}^m v^i & \nonumber \\
%\mbox{s.t. } & v^i\ge d^i-t, & i=1,\ldots,n \label{k-centrum}\nonumber \\
%& d^i\ge \sum_{j=1}^m z_{ij}& i=1,\ldots,n\nonumber \\
%& z_{ij}\geq x_j-p_{ij}x_j&i=1,\ldots,n,\;j=1,\ldots,m \nonumber \\
%& z_{ij}\geq p_{ij}-p_{ij}x_j&i=1,\ldots,n,\;j=1,\ldots,m \nonumber  \\
%& x_j\in \{0,1\} &j=1,\ldots,m \nonumber \\
%& 0\leq z_{ij}\leq 1 & i=1,\ldots,n,\;j=1,\ldots,m \nonumber \\
%& t\ge 0, \; v^i\ge 0 & i=1,\ldots,n. \nonumber
%\end{align}

\noindent When a constraint on the number $C$ of candidates to be elected must be also considered, we can add it to the above program by condition

$$ \sum_{j=1}^m x_j\le C.$$

\vskip 8 pt
\noindent An alternative valid formulation for the general $k$-sum approval voting problem can be provided by exploiting the one proposed in \cite{PuertoBlanco2014}, as stated in the following theorem.

\begin{theorem}
The following is a valid formulation for the general $k$-sum approval voting problem.
\begin{align}
\min\; & \sum_{i=1}^n u_i+ \sum_{h=1}^k v_{h} & \label{for2:obj} \\
\mbox{s.t. } & u_i+v_{h}\ge d_i & i=1,\ldots,n, \;h=1,\ldots,k \label{for2:cons1}\\
& d_i\ge \sum_{j=1}^m z_{ij}& i=1,\ldots,n \label{for2:cons2}\\
& z_{ij}\geq x_j(1-p_{ij})+p_{ij}(1-x_j) & i=1,\ldots,n,\;j=1,\ldots,m \label{for2:cons3} \\
%& z_{ij}\geq x_j-p_{ij}x_j&i=1,\ldots,n,\;j=1,\ldots,m  \\
%& z_{ij}\geq p_{ij}-p_{ij}x_j&i=1,\ldots,n,\;j=1,\ldots,m   \\
& x_j\in \{0,1\} &j=1,\ldots,m \label{for2:x_var} \\
%& 0\leq z_{ij}\leq 1 & i=1,\ldots,n,\;j=1,\ldots,m \label{for2:z_var}  \\
&  u_i\ge 0 & i=1,\ldots,n \label{for2:v_var}\\
&  v_{h}\ge 0 & h=1,\ldots,k. \label{for2:varu}
\end{align}
\end{theorem}
\begin{proof}
Consider the general formulation (\ref{general_3}). We introduce the following binary variables $y_{ih}$, $i=1,\ldots,n$ and $h=1,\ldots,k$, such that, given $x$, $y_{ih}=1$ if the distance $d_i(x)$ of voter $i$ is in position $h<k$ in the non-increasing ordering, and $y_{ih}=0$ otherwise. Following the proof in \cite{PuertoBlanco2014}, for a given vector $x$, the sum of the $k$ largest distances can be written

\begin{equation}\label{Inner_Blanco}
\begin{array}{rcll}
\sum\limits_{h=1}^k d_{\sigma_x(h)}(x)=\max & \sum\limits_{i=1}^n\sum\limits_{h=1}^k d_i y_{ih} & &\\
\\
\mbox{s.t. } & \sum\limits_{i=1}^n y_{ih}=1 & h=1,\ldots,k& \\
&\sum\limits_{h=1}^k y_{ih}=1 & i=1,\ldots,n &\\
\\
&y_{ih}\in \{0,1\} & i=1,\ldots,n,\; h=1,\ldots,k.&\\
\end{array}
\end{equation}

\noindent In fact, problem (\ref{Inner_Blanco}) is an assignment problem, so that we can relax the binary variables $0\leq y_{ih}\leq 1$. Taking the dual of (\ref{Inner_Blanco}) we obtain

\begin{equation}\label{InnerDual_Blanco}
\begin{array}{rcll}
\sum\limits_{h=1}^k d_{\sigma_x(h)}(x)=\min & \sum\limits_{i=1}^n u_i + \sum\limits_{h=1}^k v_i & &\\
\\
\mbox{s.t. } & u_i+v_{h} \geq d_i & i=1,\ldots,n,\; h=1,\ldots,k& \\
\\
&u_i,\, v_{h}\geq 0 & i=1,\ldots,n,\; h=1,\ldots,k.&\\
\end{array}
\end{equation}

\noindent To complete the proof, it suffices to add to the above dual problem the constraints (\ref{for2:cons2})-(\ref{for2:x_var}).

\end{proof}

\begin{proposition}
The formulations (\ref{GeneralProblemDis1})-(\ref{GeneralProblemDis1:const2}), (\ref{for:obj})-(\ref{Hdistance2}) and (\ref{for2:obj})-(\ref{for2:varu}) produce the same LP bound.
\end{proposition}
\begin{proof}
Let us consider a generic $\hat x\in [0,1]^m$. From our discussion above, it is clear that, fixing $\hat x$, the objective function value of all problems (\ref{GeneralProblemDis1})-(\ref{GeneralProblemDis1:const2}), (\ref{for:obj})-(\ref{Hdistance2}) and (\ref{for2:obj})-(\ref{for2:varu}) equals $\sum\limits_{h=1}^k d_{\sigma_{\hat x(h)}}(\hat x)$, where $d_{\sigma_{\hat x(1)}}(\hat x)\ge \ldots \ge d_{\sigma_{\hat x(n)}}(\hat x)$. Hence, the three problems return the same objective function value for each feasible solution of the continuous polytope and therefore this proves the claim.
\end{proof}

\section{Strengthening the formulations}\label{sec4}

The above MIP formulations are exact but still one can observe that they have some GAP at the root node of the Branch \& Bound tree (see Section \ref{s:comp-res}) although this gap is always rather small. The goal of this section is to develop some preprocessing strategies and valid inequalities that allow to improve the polyhedral description of the different formulations, and get a better bound for this GAP with a consequent gain in computational times.
\vskip 8 pt
%\subsection{Preprocessing}
\noindent First of all, we advance an easy preprocessing that allows fixing some variables either to zero or to one before the global search starts.
The rationale behind that is as follows. For a candidate $j$ to be member of at least one committee it is required that, at least for a subset $S$ of size $k<n$, there is a majority of voters that approves him/her, that is, $\gamma_j(S)\geq \lfloor k/2 \rfloor$. Therefore, if the total number of voters preferring candidate $j$, $\gamma_j$, satisfies $\gamma_j\ge n-\lfloor k/2 \rfloor$ then this candidate will be always included in any committee and thus we can set $x_j=1$. This justifies (\ref{pre:n1}). On the other hand, if candidate $j$ is only preferred by less than $\lfloor k/2 \rfloor$ voters, she/he will never be included in a $k$-sum committee and thus $x_j=0$. This justifies (\ref{pre:n2})
\begin{align}
& x_j=1&\text{if }\gamma_j\ge n-\lfloor k/2\rfloor \label{pre:n1} \\
& x_j=0&\text{if }\gamma_j\le \lfloor k/2\rfloor \label{pre:n2}.
\end{align}
Note that this preprocessing is more interesting for large values of $k$'s.
\newline

In the following we also develop a procedure for the efficient solution of the $\hat{k}$-sum approval voting problem for $\hat{k}$ fixed that works in an iterative fashion starting from $k=n$ and solving a $k$-sum approval voting problem for all $k=n,\ldots,\hat{k}$.  This procedure is based on valid inequalities involving optimal objective function values of  $k$-and-$(k+1)$-sum approval voting problems for any $k$.

\noindent First, we  analyze whether inequalities (\ref{GeneralProblem:const1})-(\ref{GeneralProblem:const}) can be adapted to the formulations described in Section \ref{sec3_2}, as valid cuts.  Clearly, they are valid inequalities but, as we will see, they do not improve such formulations.

Consider, first, formulation (\ref{for:obj})-(\ref{Hdistance2}). Note that the cuts in (\ref{GeneralProblem:const1}),  (\ref{GeneralProblem:const2}) and (\ref{GeneralProblem:const}) consist in aggregated forms of constraints (\ref{Hdistance1}), (\ref{k-centrum}) and (\ref{Hdistance0}), respectively.

\begin{align*}
z_{ij}\ge p_{ij}(1-x_j) \forall i,j \Rightarrow \sum_{i\in S}z_{ij}\ge \gamma_j(S)(1-x_{j})\quad \forall j\\
z_{ij}\ge (1-p_{ij})x_j \forall i,j \Rightarrow \sum_{i\in S}z_{ij}\ge (k-\gamma_j(S))x_{j}\quad \forall j
\end{align*}
Hence, the use of them as valid inequalities is not useful. Furthermore $\sum\limits_{j=1}^m\sum\limits_{i\in S}z_{ij}\le kt+\sum\limits_{i\in S} v_i$ can be obtained by means of a natural aggregation of (\ref{k-centrum}) and (\ref{Hdistance0}).
\newline
%Finally, $\sum\limits_{i\in S}d_i$ is a better bound that $v$ since it happens that
%\begin{equation*}
%v_i\ge d_i-t,\; \mbox{ and } d_i\ge \sum_{j=1}^mz_{ij} \Rightarrow kt+\sum_{i\in S} v_i \ge \sum_{i\in S}\sum_{j=1}^m z_{ij}.
%\end{equation*}

%Therefore, we conclude that inequalities (\ref{GeneralProblem:const1})-(\ref{GeneralProblem:const})  are valid but not useful for formulation (\ref{for:obj})-(\ref{Hdistance2}). A similar analysis applies to formulation (\ref{for2:obj})-(\ref{for2:varu}).

%\subsubsection{Valid inequalities based on bounds on the objective function value}

In light of the above results, now we develop valid inequalities based on solutions of the $k$-sum approval voting problem for different $k$ values to be used in a strategy that solves problems for different $k$ consecutively.

In order to present the result some additional notation is required. For a given $k$, let us denote by $z(k)$ and $x(k)$ the optimal objective function value and an optimal solution of the $k$-sum approval voting, respectively.

\begin{proposition}\label{pro:cuts}
For a given instance $(N,A,P)$ of the $k$-sum approval voting problem the following inequality holds
$$z(k)\ge \frac{k}{k+1} z(k+1).$$
\end{proposition}
\begin{proof}
By definition, $z(k)$, is the sum of the $k$ largest Hamming distances of the voters' profiles with respect to $x(k)$. It means that distance in position $k+1$, $d_{(k+1)}(x(k))$, satisfies
$$0\leq d_{(k+1)}(x(k))\leq \frac{z(k)}{k}.$$
Thus, we can conclude that $x(k)$ is a feasible solution for the $(k+1)$-sum problem and moreover, there exists  an upper bound for $z(k+1)$ given by

\begin{align}
z(k+1)\le z(k)+ \frac{z(k)}{k}\label{UB}.
\end{align}

The above expression gives a lower bound for $z(k)$, provided that $z(k+1)$ is known

\begin{align}
z(k)\ge \frac{k}{k+1} z(k+1)\label{LB}.
\end{align}

\end{proof}

Since, for a given $(N,A,P)$ we can solve the different $k$-sum voting problems in any order, after the above result, it is advisable to do it following the non-increasing sequence $k=n,\dots,1$. Indeed, as shown in Section \ref{Sec2}, solving the problem for $k=n$ (i.e., the minisum problem) is polynomial. Once this solution and objective function value are found, they can be used to improve the solution for $k=n-1$ and so on.

The following is an iterative scheme for efficiently solving the $k$-sum approval voting problem following the strategies illustrated above in this section. This approach has been effectively used in our computational experiments.
\begin{enumerate}
\item
Solve the problem for $k=n$, i.e. the minisum problem. Its optimal solution, $x(n)$, is easily seen to be
\begin{align}
& x_j=1&\text{if }\gamma_j\ge n-\lfloor n/2\rfloor  \\
& x_j=0&\text{if }\gamma_j< \lfloor n/2\rfloor
\end{align}
Next, obtain $z(n)$, the optimal objective function value of this problem.
\item
From $k=n-1$ to $k = 1$ set the following valid inequalities:
$$ \frac{k}{k+1} z(k+1)\le z(k)\le z(k+1)-d_{(k+1)}(x(k)).$$
\end{enumerate}
From the discussion above, it is clear that after solving the problem with $W(k+1)$ we have the lower bound on $z(k)$ (\ref{LB}) that we can use as a valid inequality when solving the problem with weighting vector $W(k)$. On the other hand, (\ref{UB}) provides an upper bound on $z(k)$ by $z(k+1)$.
We will show in our computational experiments section that the improvements obtained by the application of these strategies are remarkable (see Section \ref{s:comp-res}).

%\subsection{Merging Hamming constraints}
%Constraints (\ref{Hdistance1}) and (\ref{Hdistance2}) are equivalent to the following
%
%\begin{align}
%& z_{ij}\geq x_j&i=1,\ldots,n,\;j=1,\ldots,m: p_{ij}=0  \\
%& z_{ij}\geq 1-x_j&i=1,\ldots,n,\;j=1,\ldots,m  : p_{ij}=1
%\end{align}
%or equivalently $z_{ij}\geq p_{ij}(1-x_j)+x_j(1-p_{ij})i=1,\ldots,n,\;j=1,\ldots,m$.

\section{Computational results\label{s:comp-res}}

This section reports on the results of an exhaustive computational test carried out on two sets of instances and our three formulations with and without improvements. We have tested data sets generated according to the scheme proposed in \cite{LeGrand_et_al_2002}. That paper distinguishes between \emph{uniform} and \emph{biased} data. The former refers to equal probability  distribution for 0 and 1 in the profiles, whereas the latter indicates different probabilities for them. Overall, we have solved 22750 instances with the different combinations for $n$, $m$, $k$ and uniform and biased data.

%{\color{red} ***** We may remove this paragraph after using the new data files ****. First of all, we have found 8 instances where the optimal value obtained with the approximation algorithm not the actual minimum. This must be explained since these values must have been computed with the approximation algorithm in Amanatidis et al, which only guarantees some tolerance whereas we can certify optimality.
%\begin{center}
%\begin{tabular}{ccccccc}
%num	&n	&m	&k	&W(n-k)	&Now	&Before\\
%3	&50	&35	&43	&W(7)	&715	&716 \\
%3	&50	&35	&44	&W(6)	&728	&729 \\
%2	&50	&40	&42	&W(8)	&788	&789 \\
%2	&50	&45	&45	&W(5)	&920	&921 \\
%2	&50	&45	&46	&W(4)	&934	&935 \\
%3	&100&40	&93	&W(7)	&1777	&1778 \\
%3	&100 &45&91	&W(9)	&1918	&1919 \\
%2	&100&55	&92	&W(8)	&2415	&2416
%\end{tabular}
%\end{center}
%}

In a preliminary analysis, we wanted to test the performance of our three formulations in instances of moderate size ($n=50$ and different values for $m=30, 35,40,45,50,55,60$), in order to make the decision of which are the formulations to be tested with larger instance sizes.  Tables \ref{tablen50} and \ref{tablen50biased} report the average results of all the 1750 instances for $n=50$ for the uniform and biased data (5 randomly generated instances for each $m$ and $k=1,...,50$). The results are organized as follows. Each row reports information for a different formulation that is indicated by its references. For each formulation we include information about average and maximum time (Time(s)), average and maximum percent gap at the root node (GAP (\%)), average and maximum number of nodes in the searching tree (Nodes), percentage of instances solved  at the root node (\%Solved root) and percentage of binary variables fixed with the preprocessing (\% Fixed).

\begin{table}[h]
\begin{center}
{\small
\begin{tabular}{c|cc|cc|cc|c|c}
&  \multicolumn{2}{c|}{Time(s)} & \multicolumn{2}{c|}{GAP (\%)} & \multicolumn{2}{c|}{Nodes} & \%Solved & \%Fixed\\
Form. &Avg  &	Max 	& Avg 	&Max 	& Avg 	& Max 	& root  \\ \hline
(\ref{GeneralProblemDis1})-(\ref{GeneralProblemDis1:const2})& 65.63 & 4396.86 & 0.22 & 2.63& 7388.30 & 1092091.00& 15.89&13.92\\
(\ref{for:obj})-(\ref{Hdistance2})& 0.13& 2.24&	0.22& 2.63&	276.67&	 23170&	 62.74&13.92\\
(\ref{for2:obj})-(\ref{for2:varu})& 0.58& 19.99& 0.22&2.63& 1128.03& 158152.00 & 54.17&13.92\\
\end{tabular}
}
\end{center}
\caption{Summary for $n=50$ for uniform data\label{tablen50}}
\end{table}

\begin{table}[h]
\begin{center}
{\small
\begin{tabular}{c|cc|cc|cc|c|c}
&  \multicolumn{2}{c|}{Time(s)} & \multicolumn{2}{c|}{GAP (\%)} & \multicolumn{2}{c|}{Nodes} & \%Solved & \%Fixed\\
Form. &Avg  &	Max 	& Avg 	&Max 	& Avg 	& Max 	& root  \\ \hline
(\ref{GeneralProblemDis1})-(\ref{GeneralProblemDis1:const2})&12.34	&	 258.85	 &	 0.35	&	3.12	&	973.61	&	52028	&	 17.83&	29.42\\
 (\ref{for:obj})-(\ref{Hdistance2})& 0.06	&	0.40	&	0.35	&	 3.13	 &	 15.76	&	959	&	68.91&29.42	\\
 (\ref{for2:obj})-(\ref{for2:varu})& 0.17	&	1.23	&	0.35	&	 3.13	 &	 71.33	&	10914	&	54.86&	29.42\\
\end{tabular}
}
\end{center}
\caption{Summary for $n=50$ for biased data \label{tablen50biased}}
\end{table}

For the formulation (\ref{GeneralProblemDis1})-(\ref{GeneralProblemDis1:const2}), we also report information on the average and maximum number of cuts ($\#$ Cuts), and the average and maximum number of cuts in each node ($\#$ Cuts node). This information is relevant to understand the number of constraints, out of the exponentially many in the formulation, which are needed to have a valid representation of the problem in each node of the Branch \& Bound tree.

\begin{table}
\begin{center}
\begin{tabular}{c|cccc}
Data & Avg \# Cuts& Max \# Cuts & Avg \# Cuts node & Max \# Cuts node \\ \hline Uniform &97963.15 & 9428791 & 124.43 & 463\\
Biased & 7390.41 & 315094 & 63.80 & 444
\end{tabular}
\end{center}
\caption{Valid inequalities generated in a Branch \& Bound tree for solving formulation (\ref{GeneralProblemDis1})-(\ref{GeneralProblemDis1:const2}) with $n=50$ and for uniform and biased data \label{tablen50Cuts}}
\end{table}

%\begin{center}
%\bf
%\begin{tabular}{cccc}
%Avg \# Cuts& Max \# Cuts & Avg \# Cuts node & Max \# Cuts node \\
%\hline
%7390.41 & 315094 & 63.80 & 444
%\end{tabular}
%\end{center}

We have also tested empirically,  see Table \ref{tablen50Cuts}, that the gap at the root node coincides for the three formulations. This confirms that the three formulations are equivalent in terms of LP gap  and reports about the rather small integrality gaps of these formulations.
\bigskip

The conclusion of the above tables is that formulation (\ref{GeneralProblemDis1})-(\ref{GeneralProblemDis1:const2}) is as stronger as the other two in terms of gap, but its performance is inferior in terms of time and number of problems solved at the root node. For this reason, we have decided to carry out the final test for larger instances only with formulations (\ref{for:obj})-(\ref{Hdistance2}) and (\ref{for2:obj})-(\ref{for2:varu}).
\bigskip

Next, we compare the performance of formulations (\ref{for:obj})-(\ref{Hdistance2}) and (\ref{for2:obj})-(\ref{for2:varu}) for the instances with $n=100$. Tables \ref{tablen100} and \ref{tablen100biased} report our results for the two types of data sets i.e., uniform and biased. All the information is organized as in previous Tables \ref{tablen50} and \ref{tablen50biased}.

\begin{table}[h]
\begin{center}
\begin{tabular}{rr|rr|rr|rr|c|c}
                  &&\multicolumn{2}{c|}{Time(s)}&\multicolumn{2}{c|}{GAP (\%)}&\multicolumn{2}{c|}{Nodes}&\%Solved& \% Fixed \\
			&Form.&	Avg	&	Max&	Avg 	&	Max	&	Avg&	Max	& root&	 \\
\hline
$m$=30	&	(\ref{for2:obj})-(\ref{for2:varu})	&	3.35	&	23.66	&	 0.17	 &	 0.84	&	1565.53	&	 28751	&	46.0 & 9.32	\\
	&	(\ref{for:obj})-(\ref{Hdistance2})	&	0.37	&	4.04	&	 0.17	 &	 0.84	&	1488.07	&	 36793	 &	46.4 &	 9.32\\
\hline
$m$=35	&	(\ref{for2:obj})-(\ref{for2:varu})&	4.39	&	55.32	&	 0.17	 &	 2.50	&	2695.97	&	 126951	&	43.8 & 8.87	 \\
	&	(\ref{for:obj})-(\ref{Hdistance2})	&	0.48	&	10.47	&	 0.17	 &	 2.50	&	2334.13	&	 96381	 &	44.8 &8.87	 \\
\hline
$m$=40	&	(\ref{for2:obj})-(\ref{for2:varu})&	5.40	&	93.11	&	 0.14	 &	 1.14	&	4149.20	&	 158828	&	44.4 & 8.42	 \\
	&	(\ref{for:obj})-(\ref{Hdistance2})	&	0.70	&	12.53	&	 0.14	 &	 1.14	&	3674.94	&	 121300	 &	45.0 &8.42	 \\
\hline
$m$=45	&	(\ref{for2:obj})-(\ref{for2:varu})	&	13.40	&	251.75	&	 0.13	 &	 2.00	&	11463.60	 &	 657000	&	42.6 & 9.44	 \\
	&	(\ref{for:obj})-(\ref{Hdistance2})	&	1.56	&	65.41	&	 0.13	 &	 2.00	&	9924.58	&	 578547	 &	43.6 & 9.44	 \\
\hline
$m$=50	&	(\ref{for2:obj})-(\ref{for2:varu})	&	18.30	&	532.31	&	 0.12	 &	 0.93	&	19495.08	 &	 781787	&	43.6 & 12.38	 \\
	&	(\ref{for:obj})-(\ref{Hdistance2})	&	2.48	&	78.21	&	 0.12	 &	 0.93	&	17922.55	&	 605694	&	45.2 &	 12.38\\
\hline
$m$=55	&	(\ref{for2:obj})-(\ref{for2:varu})	&	57.27	&	1652.28	&	 0.10	 &	 0.49	&	53060.66	 &	 1875622	&	 41.0 & 7.55	\\
	&	(\ref{for:obj})-(\ref{Hdistance2})	&	5.12	&	131.90	&	 0.10	 &	 0.49	&	45278.16	&	 1256747	&	40.2	 &7.55 \\
\hline
$m$=60	&	(\ref{for2:obj})-(\ref{for2:varu})	&	41.62	&	2306.34	&	 0.11	 &	 1.56	&	68105.88	 &	 7640315	&	 40.0& 10.30	\\
	&	(\ref{for:obj})-(\ref{Hdistance2})	&	5.78	&	382.24	&	 0.11	 &	 1.56	&	47840.94	&	 4036987	&	39.6 &10.30	 \\
\end{tabular}
\end{center}
\caption {Summary for $n$=100 for uniform data \label{tablen100}}
\end{table}

\begin{table}[h]
\begin{center}
\begin{tabular}{rr|rr|rr|rr|c|c}
                  &&\multicolumn{2}{c|}{Time(s)}&\multicolumn{2}{c|}{GAP (\%)}&\multicolumn{2}{c|}{Nodes}&\%Solved & \%Fixed \\
			&Form.&	Avg	&	Max&	Avg 	&	Max	&	Avg&	Max	& root	& \\
\hline
$m$=30	&	(\ref{for2:obj})-(\ref{for2:varu})	&	0.58	&	1.96	&	 0.30	 &	 3.13	&	48.83	&	 2215	&	48.2 & 28.33	\\
		&(\ref{for:obj})-(\ref{Hdistance2})	&	0.08	&	0.42	&	 0.30	 &	 3.13	&	44.13	&	 2013	&	48.6 &28.33	 \\
\hline
$m$=35	&	(\ref{for2:obj})-(\ref{for2:varu})&	0.63	&	2.02	&	 0.28	 &	 2.63	&	77.93	&	 4147	&	43.6 & 30.07	 \\
		&(\ref{for:obj})-(\ref{Hdistance2})	&	0.10	&	0.93	&	 0.28	 &	 2.63	&	81.37	&	 6788	&	45.6 &30.07	 \\
\hline
$m$=40	&	(\ref{for2:obj})-(\ref{for2:varu})	&	0.76	&	2.47	&	 0.25	 &	 2.38	&	193.00	&	 9921	&	42.0 &26.62	\\
		&(\ref{for:obj})-(\ref{Hdistance2})	&	0.13	&	0.91	&	 0.25	 &	 2.38	&	177.40	&	 8989	&	48.4 &26.62	 \\
\hline
$m$=45	&	(\ref{for2:obj})-(\ref{for2:varu})	&	0.69	&	2.23	&	 0.24	 &	 2.17	&	145.96	&	 8026	&	42.4 &27.24	\\
		&(\ref{for:obj})-(\ref{Hdistance2})	&	0.12	&	0.81	&	 0.24	 &	 2.17	&	121.17	&	 5985	&	46.0 &27.24	 \\
\hline
$m$=50	&	(\ref{for2:obj})-(\ref{for2:varu})	&	0.69	&	2.93	&	 0.23	 &	 2.00	&	263.15	&	 15166	&	44.6 &29.21	\\
		&(\ref{for:obj})-(\ref{Hdistance2})	&	0.13	&	0.93	&	 0.23	 &	 2.00	&	202.81	&	 6125	&	50.2&29.21	 \\
\hline
$m$=55	&	(\ref{for2:obj})-(\ref{for2:varu})	&	0.81	&	6.33	&	 0.21	 &	 1.23	&	497.78	&	 25672	&	 46.6&28.69	\\
		&(\ref{for:obj})-(\ref{Hdistance2})	&	0.17	&	3.33	&	 0.21	 &	 1.23	&	443.27	&	 26845	&	52.2 &28.69	 \\
\hline
$m$=60	&	(\ref{for2:obj})-(\ref{for2:varu})	&	1.06	&	39.10	&	 0.22	 &	 1.67	&	1187.28	&	 221519	&	35.4 & 28.51	\\
		&(\ref{for:obj})-(\ref{Hdistance2})	&	0.26	&	24.92	&	 0.22	 &	 1.67	&	1136.49	&	 234681	&	39.0 &28.51	 \\

\end{tabular}
\end{center}
\caption {Summary for $n$=100 for biased data} \label{tablen100biased}
\end{table}

\begin{table}[h]
\begin{center}
\begin{tabular}{rr|rr|rr|rr|c|c}
                  &&\multicolumn{2}{c|}{Time(s)}&\multicolumn{2}{c|}{GAP (\%)}&\multicolumn{2}{c|}{Nodes}&\%Solved& \%Fixed \\
			&Form.&	Avg	&	Max&	Avg 	&	Max	&	Avg&	Max	& root	& \\
\hline
$m$=30	&	(\ref{for:obj})-(\ref{Hdistance2})	&	0.44	&	9.12	&	 0.12	 &	 1.85	&	819.95	&	 35683	&	60.6 & 29.83	\\
\hline
$m$=35	&	(\ref{for:obj})-(\ref{Hdistance2})	&	2.80	&	39.15	&	 0.12	 &	 1.59	&	7962.11	&	 162887	&	38.5 & 14.09	\\
\hline
$m$=40	&	(\ref{for:obj})-(\ref{Hdistance2})	&	6.27	&	225.75	&	 0.10	 &	 1.59	&	22843.11	&	 1165252	&	 56.1& 24.90	\\
\hline
$m$=45	&	(\ref{for:obj})-(\ref{Hdistance2})&	21.63	&	985.73	&	0.09	 &	 1.33	&	76974.33	&	 4847789	&	48.6& 27.74	 \\
\hline
$m$=50	&	(\ref{for:obj})-(\ref{Hdistance2})&	42.09	&	1129.91	&	0.08	 &	 0.93	&	143305.51	&	 4607534	&	41.7& 24.76	 \\
\hline

$m$=55 & (\ref{for:obj})-(\ref{Hdistance2}) &	50.93	&	5070.65	&	0.09	 &	 1.61	&	186926.77	&	28259032	&	40.7& 20.79\\
\hline
$m$=60 & (\ref{for:obj})-(\ref{Hdistance2}) &	49.97	&	5409.05	&	0.08	 &	 0.81	&	207304.41	&	29489376	&	56.8& 31.45 \\

\end{tabular}
\end{center}
\caption {Summary for $n$=200 for biased data} \label{tablen200biased}
\end{table}

\begin{table}[h]
\begin{center}
\begin{tabular}{rr|rr|rr|rr|c|c}
                  &&\multicolumn{2}{c|}{Time(s)}&\multicolumn{2}{c|}{GAP (\%)}&\multicolumn{2}{c|}{Nodes}&\%Solved & \%Fixed \\
			&Form.&	Avg	&	Max&	Avg 	&	Max	&	Avg&	Max	& root	 \\
\hline
$m$=30 	&	(\ref{for:obj})-(\ref{Hdistance2})	&	1.20	&	11.01	&	 0.15	&	 0.93	&	4302.39	&	60713	&	28.27 &	 7.44\\
\hline
$m$=35	&	(\ref{for:obj})-(\ref{Hdistance2})	&	2.74	&	45.11	&	 0.15	&	 2.38	&	11338.78	&	237964	&	 30.80&6.97	\\
\hline
$m$=40	&	(\ref{for:obj})-(\ref{Hdistance2})	&	7.87	&	138.90	&	 0.13	&	 0.62	&	36490.51	&	768328	&	 28.80&7.43	\\
\hline
$m$=45	&	(\ref{for:obj})-(\ref{Hdistance2})	&	12.64	&	302.69	&	 0.12	&	 1.92	&	56540.63	&	1886774	&	 27.33&7.61	\\
\hline
$m$=50 	&	(\ref{for:obj})-(\ref{Hdistance2})	&	44.67	&	777.99	&	 0.11	&	 0.89	&	211780.07	&	4787806	&	 28.67&7.43	\\
\hline
$m$=55	&	(\ref{for:obj})-(\ref{Hdistance2})	&	81.99	&	4068.90	&	 0.10	&	 1.08	&	419709.27	&	27576236	&	 23.87&7.48	\\
\hline
$m$=60 	&	(\ref{for:obj})-(\ref{Hdistance2})	&	209.87	&	6664.97	&	 0.10	&	 0.76	&	1023055.10	&	36286662	&	 24.67& 6.99	\\

\end{tabular}
\end{center}
\caption {Summary for $n$=150 for uniform data} \label{tablen150unbiased}
\end{table}

The conclusions from Tables \ref{tablen100} and \ref{tablen100biased}  are the following. Formulation (\ref{for:obj})-(\ref{Hdistance2}) is one order of magnitude faster than (\ref{for2:obj})-(\ref{for2:varu}). For instance,
the average time for the largest instance sizes ($n=100, m=60$), solved with (\ref{for:obj})-(\ref{Hdistance2}), is of $5.78$ seconds and the maximum cpu time was 382.24 seconds. The same instances solved with (\ref{for2:obj})-(\ref{for2:varu}) take an average time of $41.62$ and a maximum of $2306.34$ seconds.
This fact can be explained by the smaller number of variables and constraints that are needed in formulation (\ref{for:obj})-(\ref{Hdistance2}) with respect to (\ref{for2:obj})-(\ref{for2:varu}). The remaining factors (GAP, Nodes,  \%Solved  and \% Fixed) are quite similar in both cases. In fact, as we have seen theoretically, both formulations are equivalent in terms of LP gap and they always fix the same number of binary variables. It is also very interesting to test the practical performance of a na\"ive approximation algorithm based on using the solutions of the linear relaxation, as proposed for instance in Amanatidis et al \cite{Amanatidis_et_al2015}. One can easily bound from above the empirical performance ratio, $\frac{LPval}{optval}$, of any of our formulations based on the relative gap ($\displaystyle Rgap:=\frac{optval - LPval}{optval}* 100 \%$). Indeed, $\frac{optval}{LPval}\le \frac{100}{100-Rgap}$. Actually, since the largest relative gap is below 3.13\% (see Table 4), this results, in the worst case ($n=100$, $m=30$), with an empirical performance ratio bounded above by $1.03$.

Next, we have tested our best formulations, namely (\ref{for:obj})-(\ref{Hdistance2}), in order to explore the size limit that can be solved within 7200 seconds. Tables \ref{tablen200biased} and \ref{tablen150unbiased} report our results. As can be observed in these tables, there is a difference in the performance with respect to uniform and biased data. For the uniform data we get to the time limit already for $n=150$, whereas for the biased data we are able to solve to optimality all instances for $n=200$. The reason for this clear difference rests on the fact that the preprocessing, (\ref{pre:n1}) and (\ref{pre:n2}), is much more efficient for biased data where many variables are fixed either to zero or to one. Indeed this percentage is on average of $ 24.79\%$ for biased data with $n=200$ as compared to only $7.34\%$ for the uniform data for $n=150$.

%First of all, we tested the performance of formulation  (\ref{for:obj})-(\ref{Hdistance2}). We plan to test the plain formulation without improvements and with the preprocessing and valid cuts from Proposition \ref{pro:cuts}.

%What follows are the results from this last combination with preprocessing and cuts: For all the 6750 instances that you sent us, we obtain the following

Finally, Figures \ref{fig:time-k} and \ref{fig:root-node} show, for the 35 instances with $n=100$ (5 randomly generated instances for each value of $m\in \{30,35,40,45,50,55,60\}$), the computing time for solving the problem (Figure \ref{fig:time-k}) and the number of instances solved at the root node, without branching (Figure \ref{fig:root-node}), for the different $k=1,\ldots,100$, and uniform data.  We have observed that the behavior is similar for the different values of $n$. For that reason, we have only included those corresponding to $n=100$.

\begin{figure}
\begin{subfigure}{.5\textwidth}
  \centering
  \includegraphics[width=7cm]{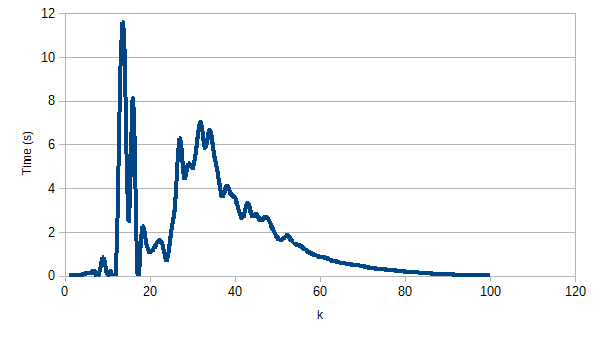}
  \caption{Time vs. $k$ for uniform data $n=100$}
  \label{fig:time-k}
\end{subfigure}%
\begin{subfigure}{.5\textwidth}
  \centering
  \includegraphics[width=7cm]{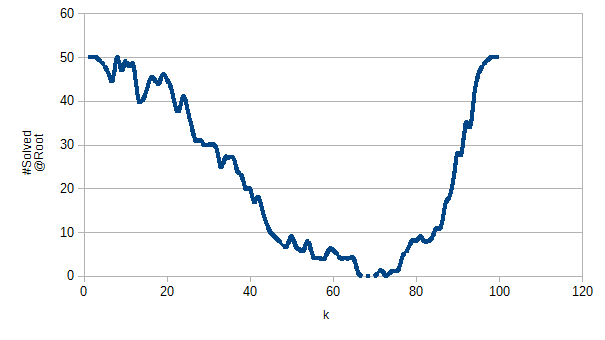}
  \caption{Solved instances at root node for uniform data $n=100$}
  \label{fig:root-node}
\end{subfigure}
\begin{subfigure}{.5\textwidth}
  \centering
  \includegraphics[width=7cm]{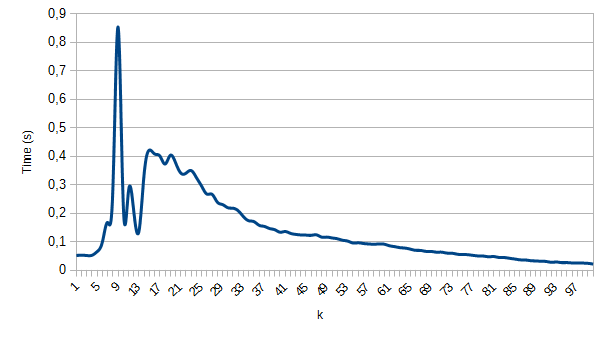}
  \caption{Time vs. $k$ for biased data $n=100$}
  \label{fig:timebiased-k}
\end{subfigure}
\begin{subfigure}{.5\textwidth}
  \centering
  \includegraphics[width=7cm]{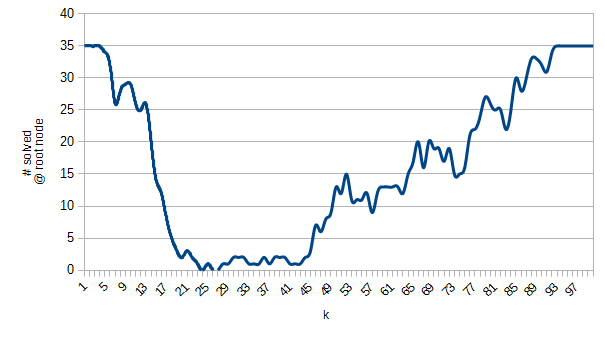}
  \caption{Solved instances at root node for biased data $n=100$}
  \label{fig:root-nodebiased}
\end{subfigure}
\caption{Plots of the efficiency of formulation (\ref{for:obj})-(\ref{Hdistance2})}
\label{fig:fig}
\end{figure}

\bigskip

%%OLD FIGURE COMMANDS ARE AFTER THE END-DOCUMENT COMMAND.

Analogously, Figures \ref{fig:timebiased-k} and \ref{fig:root-nodebiased} show for the biased data instances with $n=100$ and for the different $k=1,\ldots,100$, the computing time for solving the problem (Figure \ref{fig:timebiased-k}) and the number of instances solved at the root node, without branching (Figure \ref{fig:root-nodebiased}).

 Analyzing the figures we conclude that the behavior is rather similar for uniform and biased data. Regarding computing time for solving the problems we observe that it increases when $k$ decreases from $k=n$ until a certain threshold (which depends of the type of data, namely $k \in (15,30)$ for uniform and $k\in (9,20)$ for biased) and then the time decreases with $k$ until $k=1$. This general trend can be explained from the combinatorics of the objective function which relates to $n\choose k$. With respect to the number of instances solved at the root node, the performance is also similar. This number decreases with $k$ from $k=n$ until a certain value (which again depends on the type of data, $k\approx 70$ and $k\approx 28$, for uniform and biased data, respectively) and then it increases with $k$ up to $k=1$.

\section{Concluding Remarks}\label{sec6}

To conclude the paper, we resort to the problem of minimizing the sum of the $h$ smallest Hamming distances already introduced in Section \ref{Sec2}. In \cite{Amanatidis_et_al2015} this problem has been already considered in the approval voting application context, and the authors prove that, when the OWA vector is non-decreasing, that is, the weighting vector is of the form $M(n-h)=(0,\ldots,0,1,\ldots,1)$, with $h$ the number of ones, $1\leq h < n$, the winning committee can be found in polynomial time for a fixed value of $h$. They suggest an enumerative approach based on the solution of ${{n}\choose{n-h}}={{n}\choose{h}}$ minisum problems that is obviously not efficient even for a fixed $h$, and not polynomial if $h$ is part of the input.
As already done in Section \ref{Sec2} for the minisum problem, here we can provide a proof based on Linear Programming that formally justifies the polynomial time approach in \cite{Amanatidis_et_al2015}.
\medskip

\noindent For a fixed $h$, the general problem can be stated as

\begin{equation}\label{general_Min}
\begin{array}{lcl}
\min\limits_{x\in \{0,1\}^m} \big(\min \big\{  d_{S}(x) \; |\; S\subset \{ 1,\ldots,n \}, |S|=h \big\} \big).\\
\end{array}
\end{equation}

\noindent We can switch the two $min$ operators, thus obtaining

\begin{equation}\label{general_Min2}
\begin{array}{lcl}
\min\limits_{S\subset \{ 1,\ldots,n \}, |S|=h} \big(\min\limits_{x\in \{0,1\}^m} \big\{ d_{S}(x) \big\} \big).\\
\end{array}
\end{equation}

\noindent For a given subset $S\subset \{ 1,\ldots,n \}$, the inner minimum in problem (\ref{general_Min2}) corresponds to the minisum problem

\begin{align*}
\min & \; \sum_{j=1}^m \gamma_j(S)(1-x_j)+\sum_{j=1}^m (h-\gamma_j(S))x_j \\
\mbox{s.t. \;} & 0\le x \le 1, \nonumber
\end{align*}

\noindent which is polynomially solvable. Then, considering all the ${{n}\choose{n-h}}={{n}\choose{h}}$ subsets of cardinality $n-h$, for a fixed $h$, the problem can be solved by a sequence of LPs.
\medskip

It is worth remarking that the application of this problem to approval voting elections is meaningful. In fact, the problem can be stated as: elect a committee minimizing the sum of the $h$ smallest Hamming distances from the voters profiles. As already observed in \cite{Amanatidis_et_al2015}, the application is significant for small values of $n-h$, say $n-h=1$ or $n-h=2$. Actually, in these cases, the assumption is that the first one or two maximum distances do not play a significant role in the selection of the committee, and this is true especially when the population of voters is extremely large. The idea is that there will be always some voters whose preferences are completely disjoint from those of the majority of the others. This is, in fact, a way of considering such voters' profiles as {\em outliers}.
But, in our opinion, there are additional cases in which the application is meaningful, namely, for every choice of $h$ such that $n-h\leq \frac{n}{2}-1$. Under this condition, the approval voting problem consists of taking into account only the preferences of the absolute majority of the voters ($h \geq \frac{n}{2}+1$), with the aim of selecting the committee corresponding to the boolean vector $x^*$ for which the sum of the Hamming distances w.r.t. the $h$ considered profiles is minimized.
%Actually, the rationale is to maximize the consensus of an absolute majority of the voters (namely those $h$ voters who share a closest committee).

%Here the assumption is that, once the committee is supported by the absolute majority of voters, the objective becomes to elect the committee for which the maximum consensus is observed (remember that, once the value of $h$, and thus the one of $n-h$ is fixed, the objective function of the problem can be always seen as the sum of the preferences received by the candidates from the corresponding $h$ voters).

\noindent Note that, the two problems following the OWA approaches for approval voting studied in this paper can be seen as two different ways of facing the same problem, but giving more importance to one of the two principles that are at the basis of any democratic election. The problem with weighting vector $W(k)=(1,\ldots,1,0,\ldots,0)$ implements the idea that \emph{representation} must be maximized.

If, on the other hand, one wants to give more importance to \emph{governability}, the minimin approach (with weighting vector $M(n-h)=(0,\ldots,0,1,\ldots,1)$) can be pursued with a suitable choice of $h$, since it is able to guarantee a strong and cohesive consensus. This strength can be enforced by increasing the value of $h$. We leave the choice of which is the best voting rule for a country to its lawmakers, who, according to the specific political and social situation in which the election takes place, will be able to choose the best rule.
\vskip 12 pt
Going back to our theoretical considerations, to the best of our knowledge, the computational complexity of the minmin problem with weighting vector $M(n-h)=(0,\ldots,0,1,\ldots,1)$ when $h$ is part of the input is still an open problem. In our opinion this is an interesting issue that will be the focus of our future work.

%A valid MIP formulation for (\ref{general_Min}) can be obtained defining variables $y_i$ that assumes value $1$ if the voter $i$ is among the $h$ smallest weights and zero otherwise and %$z_{ij}=x_jy_i$. This last variable can be linearized and the results is as follows.

%\begin{align*}
%\min \; & \sum_{i=1}^n (\sum_{j=1}^m p_{ij}) y_i + \sum_{i=1}^n\sum_{j=1}^m (1-2p_{ij}) z_{ij} \\
%\mbox{s.t. } & \sum_{i=1}^n y_i=h\\
%& z_{ij}\le x_j,\quad \forall i,j,\\
%& z_{ij}\le y_i,\quad \forall i,j,\\
%& x_j\le z_{ij}-y_i+1, \quad \forall i,j,\\
%& x_j\in \{0,1\},\; y_i\in \{0,1\},\; z_{ij}\ge 0, \; \forall i,j.
%\end{align*}

%The objective function accounts for the sum of the $h$ smallest Hamming distances to the committee $x$. The first constraints ensures that exactly $h$ voters are chosen. The next three %constrains linearize the terms $x_jy_i$. Finally, the last line defines the domain of the variables.

\section*{Acknowledgement}

 This research has been partially supported by Spanish Ministry of Economy and Competitiveness/FEDER grants numbers MTM2013-46962-C02-01 and MTM2016-74983-C02-01. This paper has been written during a sabbatical leave of the second author in La Sapienza, Universit\'a di Roma whose support is also acknowledged.

\end{document}